\documentclass[a4paper]{amsart}
\usepackage{amsmath,amsthm,amssymb,latexsym,epic,bbm,comment,mathrsfs}
\usepackage{graphicx,enumerate,stmaryrd,color,tikz,ytableau}
\usepackage[all,2cell]{xy}
\xyoption{2cell}
\usetikzlibrary{patterns,snakes}

\newtheorem{theorem}{Theorem}
\newtheorem{lemma}[theorem]{Lemma}

\newtheorem{proposition}[theorem]{Proposition}
\newtheorem{conjecture}[theorem]{Conjecture}

\newtheorem{remark}[theorem]{Remark}
\usepackage[all]{xy}
\usepackage[active]{srcltx}
\usepackage[parfill]{parskip}
\usepackage{enumerate}

\begin{document}
\title[Almost all permutations and involutions are Kostant negative]
{Almost all permutations and involutions\\ are Kostant negative}

\author[Samuel Creedon and Volodymyr Mazorchuk]
{Samuel Creedon and Volodymyr Mazorchuk}

\begin{abstract}
We prove that, when $n$ goes to infinity,
Kostant's problem has negative answer for
almost all simple highest weight modules in
the principal block of the BGG category 
$\mathcal{O}$ for the 
Lie algebra  $\mathfrak{sl}_n(\mathbb{C})$.
\end{abstract}

\maketitle

\section{Introduction and the result}\label{s1}

For $n>1$, the elements of the symmetric group $S_n$
naturally index the simple highest weight modules in the 
principal block $\mathcal{O}_0$ of the Bernstein-Gelfand-Gelfand
category $\mathcal{O}$ for the Lie algebra
$\mathfrak{sl}_n(\mathbb{C})$, see \cite{BGG,Hu}.
For $w\in S_n$, we denote by $L_w$ the corresponding
simple highest weight modules.

For each such $L_w$, there is a classical question,
usually referred to as {\em Kostant's problem}, 
which asks whether the universal enveloping algebra
of $\mathfrak{sl}_n(\mathbb{C})$ surjects onto the 
algebra of adjointly-finite linear endomorphism of 
$L_w$, see \cite{Jo}. We will call $L_w$ (and also $w$)
{\em Kostant positive} if the answer is ``yes''
and {\em Kostant negative} if the answer is ``no''.
In general, the answer to Kostant's problem for 
$L_w$ is not known, however, many special cases are
settled, see \cite{Ma23} for an overview.

It is known, see \cite{MS2}, that the answer to
Kostant's problem for $L_w$ is an invariant of the 
Kazhdan-Lusztig left cell of $w$. Each such left cell
contains a unique involution. Therefore, in principle,
it is enough to solve Kostant's problem for involutions
in $S_n$. Let us denote by
\begin{itemize}
\item $\mathbf{p}_n$ the number of Kostant positive
elements in $S_n$;
\item $\mathbf{i}_n$ the number of involutions in $S_n$
(see A000085 in \cite{OEIS});
\item $\mathbf{pi}_n$ the number of Kostant positive
involutions in $S_n$.
\end{itemize}
The following two conjectures were proposed in
\cite[Subsection~6.9]{MMM}:

\begin{conjecture}\label{conj1}
Almost all elements in  $S_n$ are Kostant negative
in the sense that we have $\frac{\mathbf{p}_n}{n!}\to 0$ when
$n\to \infty$.
\end{conjecture}

\begin{conjecture}\label{conj2}
Almost all involutions in  $nS_n$ are Kostant negative
in the sense that we have $\frac{\mathbf{pi}_n}{\mathbf{i}_n}\to 0$ when
$n\to \infty$.
\end{conjecture}

The main evidence for this conjecture was a complete answer
to Kostant problem for fully commutative elements obtained
in \cite{MMM} using combinatorics of the Temperley-Lieb algebra.
This answer was detailed enough to prove the analogues of
the above conjectures for fully commutative elements. 
In the present note we confirm both conjectures and prove
the following:

\begin{theorem}\label{thm3}
Conjecture~\ref{conj1} is true.
\end{theorem}

\begin{theorem}\label{thm4}
Conjecture~\ref{conj2} is true.
\end{theorem}

Recall that a permutation $w\in S_n$ is called 
{\em consecutively $2143$-avoiding} provided that
the disjunction, over all $i=1,2,\dots,n-3$, 
of the inequalities 
\begin{displaymath}
w(i+1)<w(i)<w(i+4)<w(i+3)
\end{displaymath}
is false.
Our principal observation behind Theorems~\ref{thm3} and \ref{thm4} 
is the following necessary condition for Kostant positivity:

\begin{proposition}\label{prop5}
If $L_w$ is Kostant positive, then $w$ is consecutively
$2143$-avoiding.
\end{proposition}

This observation crystallized during our work on
two projects related to Kostant's problem for
longest elements in parabolic subgroups and
for $S_7$ which are to appear in \cite{CM1,CM2}. 
With Proposition~\ref{prop5} at hand, Theorem~\ref{thm3}
is fairly straightforward while Theorem~\ref{thm4} requires
some analytical estimates of the number of 
consecutively $2143$-avoiding involutions.
It is well-known that $2143$-avoiding involutions 
are enumerated by Motzkin numbers, see \cite{GPP},
however, we did not manage to find much information
about consecutively $2143$-avoiding involutions.

All proofs are collected in Section~\ref{s2}.

\subsection*{Acknowledgements}

The first author is partially supported by Vergstiftelsen.
The second author is partially supported by the Swedish Research Council.

\section{Proofs}\label{s2}

\subsection{Proof of Proposition~\ref{prop5}}\label{s2.1}
Our arguments here are an adaptation of the arguments in
\cite[Theorem~12]{MS} and \cite[Section~5.5]{MMM}.

For $1\leq i<n$, let $s_i=(i,i+1)$ be the elementary transposition
in $S_n$. Let $\theta_{s_i}$ be the corresponding translation through
the $s_i$-wall, this is an endofunctor of $\mathcal{O}_0$.

Assume that we are given $w\in S_n$ and $i\leq n-3$ such that
we have the consecutive $2143$-pattern in $w$ as follows:
$w(i+1)<w(i)<w(i+3)<w(i+2)$. We want to prove that the corresponding
$L_w$ is Kostant negative. Due to \cite[Theorem~8.16]{KMM},
for this it is enough to prove that 
$\theta_{s_i}\theta_{s_{i+1}}\theta_{s_{i+2}}L_w$
is isomorphic to $\theta_{s_i}L_w$.

The module $\theta_{s_{i+2}}L_w$ is non-zero as
$w(i+3)<w(i+2)$. This module has Loewy length three with 
$L_w$ being both its simple top and simple socle, cf. \cite[Proposition~46]{CMZ}. The semi-simple
Jantzen  middle contains $L_{ws_{i+1}}$ with multiplicity $1$ 
since $w$ and $ws_{i+1}$ are Bruhat neighbours
and, moreover, $ws_{i+1}>w$ due to $w(i+1)<w(i+2)$.

Next we note that $\theta_{s_{i+1}}$ kills $L_w$ as
$w(i+1)<w(i+2)$. Moreover, $\theta_{s_{i+1}}$ does not kill
$L_{ws_{i+1}}$ since $ws_{i+1}>w$. As $\theta_{s_{i+2}}L_w$ 
is indecomposable, so is $\theta_{s_{i+1}}\theta_{s_{i+2}}L_w$
by \cite[Proposition~2]{CMZ}. This implies that the latter module is
isomorphic to $\theta_{s_{i+1}}L_{ws_{i+1}}$.

Now, again, $\theta_{s_{i+1}}L_{ws_{i+1}}$ has Loewy length three with 
$L_{ws_{i+1}}$ being both its simple top and simple socle. The semi-simple
Jantzen  middle contains $L_{w}$ with multiplicity $1$ 
since $ws_{i+1}>w$ and $w$ and $ws_{i+1}$ are Bruhat neighbours.
n
Finally, we note that $\theta_{s_{i}}$ kills $L_{ws_{i+1}}$ as
$w(i)<w(i+2)$. Moreover, $\theta_{s_{i}}$ does not kill
$L_{w}$ since $w(i+1)<w(i)$n. As $\theta_{s_{i+2}}L_w$ 
is indecomposable, so is $\theta_{s_{i}}\theta_{s_{i+1}}\theta_{s_{i+2}}L_w$
by \cite[Proposition~2]{CMZ}. This implies that the latter module is
isomorphic to $\theta_{s_{i}}L_{w}$. 
Proposition~\ref{prop5} follows.

\begin{remark}\label{rem9}
{\em
Our proof of Proposition~\ref{prop5} not only shows 
Kostant negativity under the assumption of 
failing the consecutive $2143$-avoiding, but it also shows
the failure of the second K{\aa}hrstr{\"o}ms condition,
see \cite[Conjecture~1.2]{KMM}, under this assumption.
}
\end{remark}

\subsection{Proof of Theorem~\ref{thm3}}\label{s2.2}

For $k>0$, let $n>4k$ and $0\leq i\leq k-1$. For each
$w\in S_n$, there are exactly $24$ elements $u\in S_n$
with the property that $u(s)=w(s)$ for all 
$s\not\in\{4i+1,4i+2,4i+3,4i+4\}$.  Out of these 
$24$ elements, we have exactly one such that 
$u(4i+2)<u(4i+1)<u(4i+4)<u(4i+3)$. This means that, if
we choose $w$ randomly (uniformly distributed), 
it will consequently avoid $2143$ at positions
$4i+1,4i+2,4i+3,4i+4$ with probability $\frac{23}{24}$.
We call this random event $X_i$.

Clearly, for $i\neq j$, the events $X_i$ and $X_j$
are independent. Hence the probability of their
intersection, over all $i$, equals $\left(\frac{23}{24}\right)^k$.
This goes to $0$ as $k$ goes to $\infty$.
Theorem~\ref{thm3} follows.

\subsection{Proof of Theorem~\ref{thm4}}\label{s2.3}

Unfortunately, we cannot naively use the proof of Theorem~\ref{thm3}
for Theorem~\ref{thm4} as the obvious analogues of the events
$X_i$ do not really make sense and those parts which do make
sense result in events that are not independent.
Therefore we need to amend the situation in a subtle way.

Recall, see \cite[Page 64]{Kn}, that
\begin{equation}\label{eq1}
\mathbf{i}_n\sim  \mathrm{const}\cdot
n^{n/2}\cdot \exp(-n/2+\sqrt{n}).
\end{equation}

For $k>0$, we let $4k^3\leq n< 4(k+1)^3$. Define $A_1=\{1,2,3,4\}$,
$A_2=\{5,6,7,8\}$,\dots, $A_k=\{4k-3,4k-2,4k-1,4k\}$.
Let $Q_n$ be the set of all involutions $w\in S_n$
with the property that, for any $1\leq i<j\leq k$,
we have $w(A_i)\cap A_j=\varnothing$.

\begin{lemma}\label{lem6}
We have $\frac{|Q_n|}{\mathbf{i}_n}\to 1$ when $n\to \infty$. 
\end{lemma}

\begin{proof}
We do a very grave over-count for the complement $\overline{Q_n}$ of $Q_n$.
We construct an element of $\overline{Q_n}$ as follows:
\begin{itemize}
\item choose a pair $1\leq i<j\leq k$ in $\binom{k}{2}$ possible ways;
\item choose a point in $A_i$ and a point in $A_j$ in 
$4\cdot 4=16$ possible ways (these two points will be swapped by our element);
\item take any involution on the complement to this 
$2$-element set in $\mathbf{i}_{n-2}$ different ways.
\end{itemize}
This will produce all elements in $\overline{Q_n}$, however, some in a non-unique way.
This gives 
\begin{displaymath}
|\overline{Q_n}| \leq 16 \binom{k}{2}\mathbf{i}_{n-2}.
\end{displaymath}
We need to show that $\frac{|\overline{Q_n}|}{\mathbf{i}_{n}}\to 0$, when $n\to\infty$.
We use Formula~\eqref{eq1} both for $\mathbf{i}_{n-2}$ in the numerator and
for $\mathbf{i}_{n}$ in the denominator. Since $ 4k^3\leq n\leq 4(k+1)^3$,
it suffices to check that $\frac{{16}{\binom{k}{2}}{\mathbf{i}_{4k^3-2}}}{\mathbf{i}_{4k^3}}\rightarrow 0$,
when $k\to\infty$. In the ratio
$\frac{\mathbf{i}_{4k^3-2}}{\mathbf{i}_{4k^3}}$ we will get 
a $k^3$ term in the denominator, while $\binom{k}{2}$ contributes only 
with a $k^2$ term in the numerator. The claim follows.
\end{proof}

Let us pick a random element $w$ of $Q_n$ (uniformly distributed).
Denote by $X_i$ the random event that $w$ consecutively avoids
$2143$ at positions $A_i$. 

\begin{lemma}\label{lem7}
The probability of $X_i$ is at most  $\frac{23}{24}$.
\end{lemma}

\begin{proof}
We have to consider several cases depending
on $|A_i\cap w(A_i)|$.

{\bf Case~1.} $|A_i\cap w(A_i)|=4$. In this case the set of 
all involutions of $S_n$ which agree with $w$ outside 
$A_i\cup w(A_i)=A_i$ contains $10$ elements and exactly one of them does not
avoid $2143$ at the positions of $A_i$. Below  we show how
these ten elements look at the $A_i$-positions. Our convention is that a
fixed point of $w$ is shown as a singleton and a pair of points
that are swapped by $w$ are connected. The unique element that does not
avoid $2143$ is the middle  one in the second row.

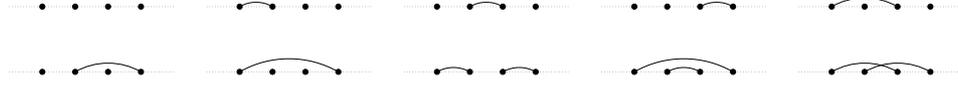
\begin{center}
\resizebox{\textwidth}{!}{
\begin{tikzpicture}
\draw[gray, thin, dotted] (0,3) -- (5,3);
\draw[black,fill=black] (1,3) circle (.5ex);
\draw[black,fill=black] (2,3) circle (.5ex);
\draw[black,fill=black] (3,3) circle (.5ex);
\draw[black,fill=black] (4,3) circle (.5ex);
\draw[gray, thin,  dotted] (6,3) -- (11,3);
\draw[black,fill=black] (7,3) circle (.5ex);
\draw[black,fill=black] (8,3) circle (.5ex);
\draw[black,thick] (8,3) arc (60:120:1);
\draw[black,fill=black] (9,3) circle (.5ex);
\draw[black,fill=black] (10,3) circle (.5ex);
\draw[gray, thin,  dotted] (12,3) -- (17,3);
\draw[black,fill=black] (13,3) circle (.5ex);
\draw[black,fill=black] (14,3) circle (.5ex);
\draw[black,fill=black] (15,3) circle (.5ex);
\draw[black,thick] (15,3) arc (60:120:1);
\draw[black,fill=black] (16,3) circle (.5ex);
\draw[gray, thin,  dotted] (18,3) -- (23,3);
\draw[black,fill=black] (19,3) circle (.5ex);
\draw[black,fill=black] (20,3) circle (.5ex);
\draw[black,fill=black] (21,3) circle (.5ex);
\draw[black,fill=black] (22,3) circle (.5ex);
\draw[black,thick] (22,3) arc (60:120:1);
\draw[gray, thin,  dotted] (24,3) -- (29,3);
\draw[black,fill=black] (25,3) circle (.5ex);
\draw[black,fill=black] (26,3) circle (.5ex);
\draw[black,fill=black] (27,3) circle (.5ex);
\draw[black,thick] (27,3) arc (60:120:2);
\draw[black,fill=black] (28,3) circle (.5ex);
\draw[gray, thin,  dotted] (0,1) -- (5,1);
\draw[black,fill=black] (1,1) circle (.5ex);
\draw[black,fill=black] (2,1) circle (.5ex);
\draw[black,fill=black] (3,1) circle (.5ex);
\draw[black,fill=black] (4,1) circle (.5ex);
\draw[black,thick] (4,1) arc (60:120:2);
\draw[gray, thin,  dotted] (6,1) -- (11,1);
\draw[black,fill=black] (7,1) circle (.5ex);
\draw[black,fill=black] (8,1) circle (.5ex);
\draw[black,fill=black] (9,1) circle (.5ex);
\draw[black,fill=black] (10,1) circle (.5ex);
\draw[black,thick] (10,1) arc (60:120:3);
\draw[gray, thin,  dotted] (12,1) -- (17,1);
\draw[black,fill=black] (13,1) circle (.5ex);
\draw[black,fill=black] (14,1) circle (.5ex);
\draw[black,thick] (14,1) arc (60:120:1);
\draw[black,fill=black] (15,1) circle (.5ex);
\draw[black,fill=black] (16,1) circle (.5ex);
\draw[black,thick] (16,1) arc (60:120:1);
\draw[gray, thin,  dotted] (18,1) -- (23,1);
\draw[black,fill=black] (19,1) circle (.5ex);
\draw[black,fill=black] (20,1) circle (.5ex);
\draw[black,fill=black] (21,1) circle (.5ex);
\draw[black,thick] (21,1) arc (60:120:1);
\draw[black,fill=black] (22,1) circle (.5ex);
\draw[black,thick] (22,1) arc (60:120:3);
\draw[gray, thin,  dotted] (24,1) -- (29,1);
\draw[black,fill=black] (25,1) circle (.5ex);
\draw[black,fill=black] (26,1) circle (.5ex);
\draw[black,fill=black] (27,1) circle (.5ex);
\draw[black,thick] (27,1) arc (60:120:2);
\draw[black,fill=black] (28,1) circle (.5ex);
\draw[black,thick] (28,1) arc (60:120:2);
\end{tikzpicture}
}
\end{center}

{\bf Case~2.} $|A_i\cap w(A_i)|=3$. Consider
the set of all involutions of $S_n$ which agree with $w$ outside 
$A_i\cup w(A_i)=A_i\cup\{r\}$ and for which $w(r)\in A_i$.
Note that, from the definition of $Q_n$ it follows that $r$
is greater than all elements of $A_i$.
This set of involutions contains $16$ elements and exactly one of them does not
avoid $2143$ at the positions of $A_i$. Below  we show how
these $16$ elements look at the $A_i$-positions. The
additional convention is that the point that is swapped
with $r$ is bigger. The unique element that does not
avoid $2143$ is  the second one in the third row.
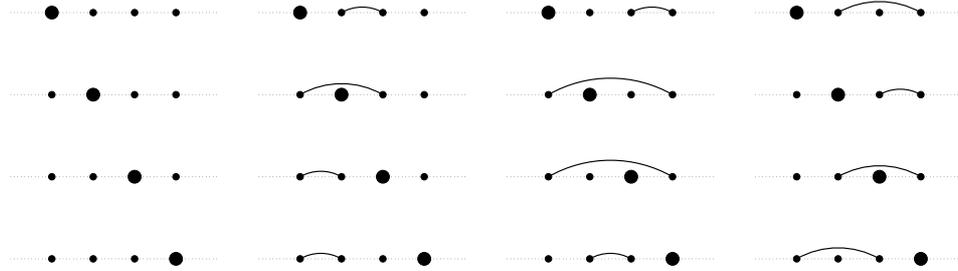
\begin{center}
\resizebox{\textwidth}{!}{
\begin{tikzpicture}
\draw[gray, thin, dotted] (0,7) -- (5,7);
\draw[black,fill=black] (1,7) circle (1ex);
\draw[black,fill=black] (2,7) circle (.5ex);
\draw[black,fill=black] (3,7) circle (.5ex);
\draw[black,fill=black] (4,7) circle (.5ex);
\draw[gray, thin,  dotted] (6,7) -- (11,7);
\draw[black,fill=black] (7,7) circle (1ex);
\draw[black,fill=black] (8,7) circle (.5ex);
\draw[black,fill=black] (9,7) circle (.5ex);
\draw[black,thick] (9,7) arc (60:120:1);
\draw[black,fill=black] (10,7) circle (.5ex);
\draw[gray, thin,  dotted] (12,7) -- (17,7);
\draw[black,fill=black] (13,7) circle (1ex);
\draw[black,fill=black] (14,7) circle (.5ex);
\draw[black,fill=black] (15,7) circle (.5ex);
\draw[black,fill=black] (16,7) circle (.5ex);
\draw[black,thick] (16,7) arc (60:120:1);
\draw[gray, thin,  dotted] (18,7) -- (23,7);
\draw[black,fill=black] (19,7) circle (1ex);
\draw[black,fill=black] (20,7) circle (.5ex);
\draw[black,fill=black] (21,7) circle (.5ex);
\draw[black,fill=black] (22,7) circle (.5ex);
\draw[black,thick] (22,7) arc (60:120:2);
\draw[gray, thin,  dotted] (0,5) -- (5,5);
\draw[black,fill=black] (1,5) circle (.5ex);
\draw[black,fill=black] (2,5) circle (1ex);
\draw[black,fill=black] (3,5) circle (.5ex);
\draw[black,fill=black] (4,5) circle (.5ex);
\draw[gray, thin,  dotted] (6,5) -- (11,5);
\draw[black,fill=black] (7,5) circle (.5ex);
\draw[black,fill=black] (8,5) circle (1ex);
\draw[black,fill=black] (9,5) circle (.5ex);
\draw[black,thick] (9,5) arc (60:120:2);
\draw[black,fill=black] (10,5) circle (.5ex);
\draw[gray, thin,  dotted] (12,5) -- (17,5);
\draw[black,fill=black] (13,5) circle (.5ex);
\draw[black,fill=black] (14,5) circle (1ex);
\draw[black,fill=black] (15,5) circle (.5ex);
\draw[black,fill=black] (16,5) circle (.5ex);
\draw[black,thick] (16,5) arc (60:120:3);
\draw[gray, thin,  dotted] (18,5) -- (23,5);
\draw[black,fill=black] (19,5) circle (.5ex);
\draw[black,fill=black] (20,5) circle (1ex);
\draw[black,fill=black] (21,5) circle (.5ex);
\draw[black,fill=black] (22,5) circle (.5ex);
\draw[black,thick] (22,5) arc (60:120:1);
\draw[gray, thin, dotted] (0,3) -- (5,3);
\draw[black,fill=black] (1,3) circle (.5ex);
\draw[black,fill=black] (2,3) circle (.5ex);
\draw[black,fill=black] (3,3) circle (1ex);
\draw[black,fill=black] (4,3) circle (.5ex);
\draw[gray, thin,  dotted] (6,3) -- (11,3);
\draw[black,fill=black] (7,3) circle (.5ex);
\draw[black,fill=black] (8,3) circle (.5ex);
\draw[black,thick] (8,3) arc (60:120:1);
\draw[black,fill=black] (9,3) circle (1ex);
\draw[black,fill=black] (10,3) circle (.5ex);
\draw[gray, thin,  dotted] (12,3) -- (17,3);
\draw[black,fill=black] (13,3) circle (.5ex);
\draw[black,fill=black] (14,3) circle (.5ex);
\draw[black,fill=black] (15,3) circle (1ex);
\draw[black,fill=black] (16,3) circle (.5ex);
\draw[black,thick] (16,3) arc (60:120:3);
\draw[gray, thin,  dotted] (18,3) -- (23,3);
\draw[black,fill=black] (19,3) circle (.5ex);
\draw[black,fill=black] (20,3) circle (.5ex);
\draw[black,fill=black] (21,3) circle (1ex);
\draw[black,fill=black] (22,3) circle (.5ex);
\draw[black,thick] (22,3) arc (60:120:2);
\draw[gray, thin,  dotted] (0,1) -- (5,1);
\draw[black,fill=black] (1,1) circle (.5ex);
\draw[black,fill=black] (2,1) circle (.5ex);
\draw[black,fill=black] (3,1) circle (.5ex);
\draw[black,fill=black] (4,1) circle (1ex);
\draw[gray, thin,  dotted] (6,1) -- (11,1);
\draw[black,fill=black] (7,1) circle (.5ex);
\draw[black,fill=black] (8,1) circle (.5ex);
\draw[black,thick] (8,1) arc (60:120:1);
\draw[black,fill=black] (9,1) circle (.5ex);
\draw[black,fill=black] (10,1) circle (1ex);
\draw[gray, thin,  dotted] (12,1) -- (17,1);
\draw[black,fill=black] (13,1) circle (.5ex);
\draw[black,fill=black] (14,1) circle (.5ex);
\draw[black,fill=black] (15,1) circle (.5ex);
\draw[black,thick] (15,1) arc (60:120:1);
\draw[black,fill=black] (16,1) circle (1ex);
\draw[gray, thin,  dotted] (18,1) -- (23,1);
\draw[black,fill=black] (19,1) circle (.5ex);
\draw[black,fill=black] (20,1) circle (.5ex);
\draw[black,fill=black] (21,1) circle (.5ex);
\draw[black,thick] (21,1) arc (60:120:2);
\draw[black,fill=black] (22,1) circle (1ex);
\end{tikzpicture}
}
\end{center}

{\bf Case~3.} $|A_i\cap w(A_i)|=2$. Consider
the set of all involutions of $S_n$ which agree with $w$ outside 
$A_i\cup w(A_i)=A_i\cup\{r,s\}$ and for which $w(r),w(s)\in A_i$.
Note that, because of the definition of $Q_n$, we may assume $A_i<r<s$. 
This set of involution contains $24$ elements and exactly one of them does not
avoid $2143$ at the positions of $A_i$. Below  we show how
these elements look like at the $A_i$-positions. We list the elements
up to the information how the points of $A_i$ are connected
to $r$ and $s$ (hence each of our twelve diagrams correspond to
two elements). The unique element that does not
avoid $2143$  is
the last one in the last row for which the connection to $r$ and $s$
reverses the natural order. 
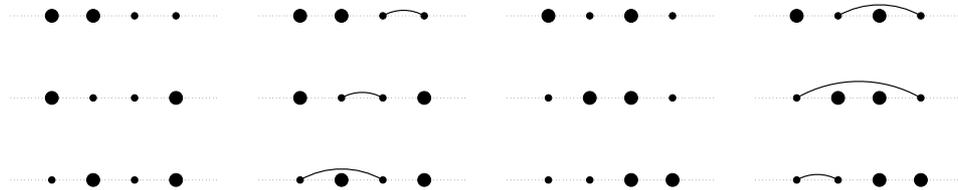
\begin{center}
\resizebox{\textwidth}{!}{
\begin{tikzpicture}
\draw[gray, thin,  dotted] (0,5) -- (5,5);
\draw[black,fill=black] (1,5) circle (1ex);
\draw[black,fill=black] (2,5) circle (1ex);
\draw[black,fill=black] (3,5) circle (.5ex);
\draw[black,fill=black] (4,5) circle (.5ex);
\draw[gray, thin,  dotted] (6,5) -- (11,5);
\draw[black,fill=black] (7,5) circle (1ex);
\draw[black,fill=black] (8,5) circle (1ex);
\draw[black,fill=black] (9,5) circle (.5ex);
\draw[black,fill=black] (10,5) circle (.5ex);
\draw[black,thick] (10,5) arc (60:120:1);
\draw[gray, thin,  dotted] (12,5) -- (17,5);
\draw[black,fill=black] (13,5) circle (1ex);
\draw[black,fill=black] (14,5) circle (.5ex);
\draw[black,fill=black] (15,5) circle (1ex);
\draw[black,fill=black] (16,5) circle (.5ex);
\draw[gray, thin,  dotted] (18,5) -- (23,5);
\draw[black,fill=black] (19,5) circle (1ex);
\draw[black,fill=black] (20,5) circle (.5ex);
\draw[black,fill=black] (21,5) circle (1ex);
\draw[black,fill=black] (22,5) circle (.5ex);
\draw[black,thick] (22,5) arc (60:120:2);
\draw[gray, thin,  dotted] (0,3) -- (5,3);
\draw[black,fill=black] (1,3) circle (1ex);
\draw[black,fill=black] (2,3) circle (.5ex);
\draw[black,fill=black] (3,3) circle (.5ex);
\draw[black,fill=black] (4,3) circle (1ex);
\draw[gray, thin,  dotted] (6,3) -- (11,3);
\draw[black,fill=black] (7,3) circle (1ex);
\draw[black,fill=black] (8,3) circle (.5ex);
\draw[black,fill=black] (9,3) circle (.5ex);
\draw[black,thick] (9,3) arc (60:120:1);
\draw[black,fill=black] (10,3) circle (1ex);
\draw[gray, thin,  dotted] (12,3) -- (17,3);
\draw[black,fill=black] (13,3) circle (.5ex);
\draw[black,fill=black] (14,3) circle (1ex);
\draw[black,fill=black] (15,3) circle (1ex);
\draw[black,fill=black] (16,3) circle (.5ex);
\draw[gray, thin,  dotted] (18,3) -- (23,3);
\draw[black,fill=black] (19,3) circle (.5ex);
\draw[black,fill=black] (20,3) circle (1ex);
\draw[black,fill=black] (21,3) circle (1ex);
\draw[black,fill=black] (22,3) circle (.5ex);
\draw[black,thick] (22,3) arc (60:120:3);
\draw[gray, thin,  dotted] (0,1) -- (5,1);
\draw[black,fill=black] (1,1) circle (.5ex);
\draw[black,fill=black] (2,1) circle (1ex);
\draw[black,fill=black] (3,1) circle (.5ex);
\draw[black,fill=black] (4,1) circle (1ex);
\draw[gray, thin,  dotted] (6,1) -- (11,1);
\draw[black,fill=black] (7,1) circle (.5ex);
\draw[black,fill=black] (8,1) circle (1ex);
\draw[black,fill=black] (9,1) circle (.5ex);
\draw[black,thick] (9,1) arc (60:120:2);
\draw[black,fill=black] (10,1) circle (1ex);
\draw[gray, thin,  dotted] (12,1) -- (17,1);
\draw[black,fill=black] (13,1) circle (.5ex);
\draw[black,fill=black] (14,1) circle (.5ex);
\draw[black,fill=black] (15,1) circle (1ex);
\draw[black,fill=black] (16,1) circle (1ex);
\draw[gray, thin,  dotted] (18,1) -- (23,1);
\draw[black,fill=black] (19,1) circle (.5ex);
\draw[black,fill=black] (20,1) circle (.5ex);
\draw[black,thick] (20,1) arc (60:120:1);
\draw[black,fill=black] (21,1) circle (1ex);
\draw[black,fill=black] (22,1) circle (1ex);
\end{tikzpicture}
}
\end{center}

{\bf Case~4.} $|A_i\cap w(A_i)|=1$. Consider
the set of all involutions of $S_n$ which agree with $w$ outside 
$A_i\cup w(A_i)=A_i\cup\{r,s,t\}$ and for which 
$w(r),w(s),w(t)\in A_i$. 
Note that, because of the definition of $Q_n$, we may assume $A_i<r<s<t$. 
This set of involutions contains $24$ elements and
exactly one of them does not avoid $2143$ at the positions of $A_i$.
Below  we show how these elements look like at the $A_i$-positions. 
We list the elements
up to the information on how the points of $A_i$ are connected
to $r$, $s$ and $t$. The unique element that does not
avoid $2143$ corresponds to the third diagram for which $4i-3$ maps to $r$,
then $4i$ maps to $s$ and, finally, $4i-1$ maps to $t$.
\begin{center}
\resizebox{\textwidth}{!}{
\begin{tikzpicture}
\draw[gray, thin,  dotted] (0,1) -- (5,1);
\draw[black,fill=black] (1,1) circle (1ex);
\draw[black,fill=black] (2,1) circle (1ex);
\draw[black,fill=black] (3,1) circle (1ex);
\draw[black,fill=black] (4,1) circle (.5ex);
\draw[gray, thin,  dotted] (6,1) -- (11,1);
\draw[black,fill=black] (7,1) circle (1ex);
\draw[black,fill=black] (8,1) circle (1ex);
\draw[black,fill=black] (9,1) circle (.5ex);
\draw[black,fill=black] (10,1) circle (1ex);
\draw[gray, thin,  dotted] (12,1) -- (17,1);
\draw[black,fill=black] (13,1) circle (1ex);
\draw[black,fill=black] (14,1) circle (.5ex);
\draw[black,fill=black] (15,1) circle (1ex);
\draw[black,fill=black] (16,1) circle (1ex);
\draw[gray, thin,  dotted] (18,1) -- (23,1);
\draw[black,fill=black] (19,1) circle (.5ex);
\draw[black,fill=black] (20,1) circle (1ex);
\draw[black,fill=black] (21,1) circle (1ex);
\draw[black,fill=black] (22,1) circle (1ex);
\end{tikzpicture}
}
\end{center}

{\bf Case~5.} $|A_i\cap w(A_i)|=0$. Consider
the set of all involutions of $S_n$ which agree with $w$ outside 
$A_i\cup w(A_i)=A_i\cup\{r,s,t,v\}$ and for which 
$w(r),w(s),w(t),w(v)\in A_i$. This set contains $24$ elements 
which correspond to all possible bijections between 
$A_i$ and $\{r,s,t,v\}$ (that is, our diagram
is 
\resizebox{3cm}{!}{
\begin{tikzpicture}
\draw[gray, thin,  dotted] (0,1) -- (5,1);
\draw[black,fill=black] (1,1) circle (1ex);
\draw[black,fill=black] (2,1) circle (1ex);
\draw[black,fill=black] (3,1) circle (1ex);
\draw[black,fill=black] (4,1) circle (1ex);
\end{tikzpicture}
})
and
exactly one of them does not avoid $2143$ at the positions of $A_i$.

As $\frac{9}{10}<\frac{15}{16}<\frac{23}{24}$
and the probability of $X_i$ is a convex linear 
combination of these numbers, the claim follows.
\end{proof}

\begin{lemma}\label{lem8}
The random events $X_1,X_2,\dots,X_k$ are independent.
\end{lemma}

\begin{proof}
Let $1\leq i<j\leq k$. Due to the definition of $Q_n$, 
for $w\in Q_n$, we have $w(A_i)\cap A_j=\varnothing$. 
Therefore the information on possible
$2143$-avoidance at the positions of $A_i$  does not affect
the information on possible $2143$-avoidance at the positions of $A_j$.
The claim of the lemma follows.
\end{proof}

From Lemmata~\ref{lem7} and \ref{lem8} it follows that the
probability of the intersection of the $X_i$'s
is bounded by $\left(\frac{23}{24}\right)^{k}$
and hence approaches $0$ when $k\to\infty$.
Taking Lemma~\ref{lem6} into account,
this implies Theorem~\ref{thm4}.

\vspace{2mm}

\noindent
Department of Mathematics, Uppsala University, Box. 480,
SE-75106, Uppsala, \\ SWEDEN, 
emails:
{\tt samuel.creedon\symbol{64}math.uu.se}\hspace{5mm}
{\tt mazor\symbol{64}math.uu.se}

\end{document}